\documentclass[11pt]{article}

\usepackage{enumerate}

\usepackage{mathrsfs}
\usepackage{color,latexsym,amsfonts,amssymb}
\usepackage{amsmath}
\usepackage{amsfonts}
\usepackage{diagbox}
\usepackage{graphicx}

\usepackage{titletoc}
\usepackage{latexsym}
\usepackage{multicol}
\usepackage{graphics}
\usepackage{subfigure}
\usepackage{indentfirst}
\usepackage{epsfig}
\usepackage{amsmath,amsthm,amssymb,mathrsfs,amsbsy,bm}
\usepackage[colorlinks=true, allcolors=blue]{hyperref}
\usepackage{color, xcolor} %

\numberwithin{equation}{section}

\def\pr{\textsf{P}} 
\def\ep{\textsf{E}} 
\def\Sbep{\widehat{\mathbb E}} 
\def\cSbep{\widehat{\mathcal E}} 
\def\Capc{\mathbb V} 
\def\cCapc{\mathcal V} 
\def\upCapc{\widehat{\mathbb V}} 
\def\outCapc{\widehat{\mathbb V}^{\ast}}
\def\outcCapc{\widehat{\mathcal V}^{\ast}} 
 \newcommand*{\dif}{\mathop{}\!\mathrm{d}}

\newtheorem{theorem}{Theorem}[section]
\newtheorem{definition}{Definition}[section]
\newtheorem{lemma}{Lemma}[section]

\newtheorem{proposition}{Proposition}[section]

\newtheorem{remark}{Remark}[section]

\begin{document}

\thispagestyle{plain} 

\setlength\abovedisplayskip{2pt}
\setlength\abovedisplayshortskip{0pt}
\setlength\belowdisplayskip{2pt}
\setlength\belowdisplayshortskip{0pt}

\title{\huge \bf Limit Laws of the Iterated Logarithm Under Sub-linear Expectations}
\author{Li-Xin Zhang$^{1,2}$\footnote{  Email:  stazlx@mail.zjgsu.edu.cn$^1$; stazlx@zju.edu.cn$^2$. Corresponding author}  and Yongsheng Song$^{3,4}$\footnote{  Email:  yssong@amss.ac.cn} \vspace{3mm}\\
 \small $^{1}$School of Statistics and Data Science, Zhejiang Gongshang University, Hangzhou, 310018\\
 \small $^{2}$Center for Data Science, Zhejiang University, Hangzhou, 310058\\
 \small $^{3}$State Key Laboratory of Mathematical Sciences, Academy of Mathematics and Systems Science, \\\small Chinese Academy of Sciences, Beijing, 100190\\
 \small $^{2}$School of Mathematical Sciences, University of Chinese Academy of Sciences, Beijing, 100049}
\date{}
\maketitle

\vspace{-1cm}
\begin{center}\begin{minipage}{13cm}
{\bf Abstract}:  {Let $\{Y_n; n\ge 1\}$ be a sequence of independent and identically distributed random variables with mean zero in Peng's framework of the sub-linear expectation space $(\Omega,\mathscr{H},\Sbep)$, and $S_n=\sum_{i=1}^nY_i$. In this paper, we establish a limit law of 
\begin{align*}\lim_{n\to \infty}\max_{k\le n}\frac{S_k}{\sqrt{2k \log\log n}}. 
\end{align*}
Different from the result obtained by Chen (2015) in which the limit is a constant, it is shown that under the upper capacity the limit may be prescribed as a given function of $Y_1,Y_2,\ldots$, taking values in the standard deviation interval. As a result, it is also shown that the set of limit points in the compact law of the iterated logarithm can be a symmetric random interval.  }

 
 {\bf MSC (2010) Subject classification }: 60F15, 60F05
\end{minipage}\end{center}

\section{Introduction and Notation}\label{sect1}
 \setcounter{equation}{0}
Let $\{Y_n; n\ge 1\}$ be a sequence of independent and identically distributed random variables with mean zero and variance $\sigma^2$ ($0\le\sigma<\infty$) on a probability space $(\Omega,\mathcal{F},\pr)$, and let $S_n=\sum_{j=1}^n Y_j$. The well-known law of the iterated logarithm states that
 \begin{align}\label{eq:K-LIL} \pr\bigg(\limsup_{n\to \infty} \frac{S_n}{\sqrt{2n\log\log n}}=\limsup_{n\to \infty} \frac{\max_{k\le n}S_k}{\sqrt{2n\log\log n}}=\sigma\bigg)=1, 
 \end{align}
where $\log x=\ln (x\vee {\rm e})$. Different from the upper limit "$\limsup$" above, Chen \cite{Chen2015} established the limit form of the law of the iterated logarithm:
\begin{align}\label{eq:chenLIL} \pr\bigg(\lim_{n\to \infty} \frac{\max_{k\le n}S_k/\sqrt{k}}{\sqrt{2\log\log n}}=\sigma\bigg)=1. 
\end{align}
If we let $k_n$ be the index $k\le n$ such that $S_k/\sqrt{k}$ attains its maximum, then
$$\frac{S_n}{\sqrt{2n\log\log n}} \le \frac{\max_{k\le n}S_k/\sqrt{k}}{\sqrt{2\log\log n}}\le \frac{S_{k_n}}{\sqrt{2k_n\log\log k_n}}, $$
and it is easy to verify that (\ref{eq:chenLIL}) implies (\ref{eq:K-LIL}). The purpose of this paper is to establish a limit law of the form (\ref{eq:chenLIL}) in Peng's framework of sub-linear expectations.

  We use the framework and notation of Peng \cite{Peng2008a,Peng2019}. Let $(\Omega,\mathcal{F})$ be a given measurable space, and let $\mathscr{H}$ be a linear space of real functions defined on $(\Omega,\mathcal{F})$ satisfying the condition that if $X_1,\cdots, X_n\in\mathscr{H}$, then $\varphi(X_1,\cdots,X_n)\in\mathscr{H}$ for any $\varphi\in C_{l,Lip}(\mathbb{R}^n)$, where $C_{l,Lip}(\mathbb{R}^n)$ denotes the linear space of locally Lipschitz functions $\varphi$ satisfying the following condition: $\varphi\in C_{l,Lip}(\mathbb{R}^n)$ if and only if there exists a constant $C>0$ and $m\in\mathbb{N}$, depending on $\varphi$, such that
$$\vert \varphi(\bm{x})-\varphi(\bm{y}) \vert\leq C(1+\vert \bm{x}\vert^m+\vert\bm{y}\vert^m)\vert\bm{x}-\bm{y}\vert,\quad\forall\bm{x},\bm{y}\in\mathbb{R}^n.$$
$\mathscr{H}$ is regarded as the space of "random variables"; in this case, we write $X\in\mathscr{H}$. We also use $C_{b,Lip}(\mathbb{R}^n)$ to denote the space of bounded Lipschitz functions, and $C_b(\mathbb R^n)$ to denote the space of bounded continuous functions.
\begin{definition}
A sub-linear expectation $\Sbep :\mathscr{H}\rightarrow\bar{\mathbb{R}}$ on the space $\mathscr{H}$ is a function satisfying the following properties: for all $X,Y\in\mathscr{H}$,
\begin{itemize}
\item[(a)] If $X \geq Y$, then $\Sbep [X] \geq \Sbep [Y]$;
\item[(b)] $\Sbep [c]=c$;
\item[(c)] When $\Sbep [X]+\Sbep [Y]$ is not of the form $+\infty-\infty$ or $-\infty+\infty$, we have $\Sbep [X+Y] \leq \Sbep [X]+\Sbep [Y]$;
\item[(d)] For $\lambda>0$, we have $\Sbep [\lambda X]=\lambda\Sbep [X]$. \end{itemize}
Here, $\bar{\mathbb{R}}=[-\infty,\infty]$, and $0\cdot\infty$ is defined as $0$. The triple $(\Omega,\mathscr{H},\Sbep)$ is called a sub-linear expectation space. Given a sub-linear expectation $\Sbep$, $\cSbep$ denotes its conjugate expectation, i.e., $$ \cSbep [X]:=-\Sbep [-X],\quad\forall X\in\mathscr{H}. $$
\end{definition}
\begin{definition}
\begin{itemize}
\item[(i)](Identical distribution) Let $\bm{X}_1$ and $\bm{X}_2$ be two $n$-dimensional random vectors defined on sub-linear expectation spaces $(\Omega_1,\mathscr{H}_1,\Sbep_1)$ and $(\Omega_2,\mathscr{H}_2,\Sbep_2)$, respectively. If for any $\varphi\in C_{b,Lip}(\mathbb{R}^n)$ we have
$$ \Sbep_1[\varphi(\bm{X}_1)]=\Sbep_2[\varphi(\bm{X}_2)], $$
then they are said to be identically distributed, denoted by $\bm{X}_1\overset{d}{=}\bm{X}_2$. If $X_i\overset{d}{=}X_1$ for each $i\geq 1$, then the sequence of random variables $\{X_n;n\geq 1\}$ is said to be identically distributed.
 \item[(ii)](Independence) In a sub-linear expectation space $(\Omega,\mathscr{H},\Sbep)$, a random vector $\bm{Y}=(Y_1,\cdots,Y_n)$, $Y_i\in\mathscr{H}$ is said to be independent of another random vector $\bm{X}=(X_1,\cdots,X_m)$, $X_i\in\mathscr{H}$ under $\Sbep$, if for every test function $\varphi\in C_{b,Lip}(\mathbb{R}^m\times\mathbb{R}^n)$ we have
     $$\Sbep [\varphi(\bm{X},\bm{Y})]=\Sbep [\Sbep [\varphi(\bm{x},\bm{Y})]\vert_{\bm{x}=\bm{X}}].$$

  If for each $i\geq 1$, $X_{i+1}$ is independent of $(X_1,\cdots,X_i)$, then the sequence of random variables $\{X_n;n\geq 1\}$ is said to be independent.
\end{itemize}
\end{definition}
\begin{remark} In the definitions of identical distribution and independence, the range of test functions $\varphi\in C_{b,Lip}(\mathbb{R}^n)$ here is slightly different from Peng's \cite{Peng2008a,Peng2019} $\varphi\in C_{l,Lip}(\mathbb{R}^n)$. This is to ensure that the sub-linear expectations of the involved functions exist, avoiding the discussion of existence of expectations.
\end{remark}

  Let $(\Omega, \mathscr{H}, \Sbep)$ be a sub-linear expectation space. We assume that $(\Capc,\cCapc)$ is a pair of capacities with the following properties:
\begin{align}
\text{For any } f\leq I_A\leq g,\; f,g\in\mathscr{H}\text{ and } A\in\mathcal{F} \text{ we have } \;\Sbep [f]\leq\Capc(A)\leq\Sbep [g], \label{eq1.3} \end{align}
$$ \Capc \text{ is sub-additive, i.e., for any } A, B\in \mathcal F, \text{ we have } \Capc(A\bigcup B)\le \Capc(A)+\Capc(B), $$
and $\cCapc(A):=1-\Capc(A^c)$, $A\in\mathcal{F}$. We call $\Capc$ and $\cCapc$ the upper capacity and lower capacity, respectively.

Furthermore, we define the Choquet integrals/expectations $(C_{\Capc}, C_{\cCapc})$ by the formula:
$$
C_V[X]=\int_0^{\infty}V(X\geq t)\dif t+\int_{-\infty}^0[V(X\geq t)-1]\dif t, $$
where $V$ is replaced by $\Capc$ and $\cCapc$, respectively. If $\Capc$ on the sub-linear expectation space $(\Omega,\mathscr{H},\Sbep)$ and $\widetilde{\mathbb{V}}$ on the sub-linear expectation space $(\widetilde{\Omega},\widetilde{\mathscr{H}},\widetilde{\mathbb{E}})$ are two capacities satisfying property (\ref{eq1.3}), then for any random variables $X\in\mathscr{H}$ and $\widetilde{X}\in\widetilde{\mathscr{H}}$ satisfying $X\overset{d}{=}\widetilde{X}$, we have
$$
\Capc(X\geq x+\epsilon)\leq\widetilde{V}(\widetilde{X}\geq x)\leq\Capc(X\geq x-\epsilon)\quad \forall \enspace x,\epsilon>0$$
and
$
C_{\Capc}[X]=C_{\widetilde{\Capc}}[\widetilde{X}].
$

Usually, we choose $(\Capc,\cCapc)$ as
\begin{align}\label{eq1.5}
\widehat{\Capc}(A):=\inf\{\Sbep [\xi]:I_A\leq\xi,\xi\in\mathscr{H}\},\quad \widehat{\cCapc}(A)=1-\widehat{\Capc}(A^c), \enspace\forall A\in\mathcal{F}.
\end{align}
When there exists a family of probability measures $\mathscr{P}$ on $(\Omega,\mathscr{F})$ such that
\begin{align}\label{eq1.7} \Sbep[X]=\sup_{P\in \mathscr{P}}E_P[X]=:\sup_{P\in \mathscr{P}}\int XdP,
\end{align}
$\Capc$ can be defined as
\begin{align}\label{eq1.6} \Capc(A)=\sup_{P\in \mathscr{P}}P(A).
\end{align}
We denote this capacity by $\Capc^{\mathscr{P}}$, and $\cCapc^{\mathscr{P}}(A)=1-\Capc^{\mathscr{P}}(A)$. Obviously, $\Capc^{\mathscr{P}}$ is countably sub-additive, i.e., for any $A_n\in \mathscr{F}$, $\Capc^{\mathscr{P}}\big(\bigcup_{n=1}^{\infty}A_n\big)\le \sum_{n=1}^{\infty}\Capc^{\mathscr{P}}(A_n)$.
Since $\widehat{\Capc}$ may not be countably sub-additive, the Borel-Cantelli lemma does not hold. We consider its countably sub-additive extension $\outCapc$, defined by
\begin{align}
\outCapc(A):=\inf\bigg\{\sum_{n=1}^{\infty}
\widehat{\Capc}(A_n):A\subset\bigcup_{n=1}^{\infty}A_n\bigg\},
\outcCapc(A)=1-\outCapc(A^c),\quad A\in\mathcal{F}.
\end{align}
If $\Sbep$ is of the form (\ref{eq1.7}), then
$$ \Capc^{\mathscr{P}}(A)=\sup_{P\in \mathscr{P}}P(A)\le \outCapc(A)\le \upCapc(A), \;\; A\in\mathcal{F}. $$

As shown in Zhang \cite{Zhang2016}, $\outCapc$ is countably sub-additive, and $\outCapc(A) \leq \widehat{\Capc}(A)$. Moreover, $\widehat{\Capc}(A)$ (respectively $\outCapc$) is the largest sub-additive (respectively, countably sub-additive) capacity, i.e., if $V$ is also a sub-additive (respectively, countably sub-additive) capacity satisfying $V(A) \leq \Sbep [g]$ for any $A$ and $g$ with $I_A \leq g \in \mathscr{H}$, then $V(A) \leq \widehat{\Capc}(A)$ (respectively $V(A) \leq \outCapc(A)$).

 Finally, we introduce the (CC) (Countable-dimensional Compactness) condition proposed in Zhang \cite{Zhang2021}. A sub-linear expectation $\Sbep$ is said to satisfy condition (CC) if
\begin{align} \label{eqexpressbyP}
\Sbep [X]=\sup_{P\in\mathscr{P}}P[X],\enspace X\in\mathscr{H}_b=:\{X\in\mathscr{H} : X\enspace \text{is bounded}\},
\end{align}
where $\mathscr{P}$ is a family of probability measures on $(\Omega,\sigma(\mathscr{H}))$ that is countably-dimensional weakly compact, i.e., for any $Y_1, Y_2,\cdots\in\mathscr{H}_b$ and any sequence of probability measures $\{P_n\}\subset\mathscr{P}$, there exists a subsequence $\{n_k\}$ and a probability measure $P\in\mathscr{P}$, such that for all $\varphi\in C_{b,Lip}(\mathbb{R}^d), d\geq1$,
$$
\lim_{k\rightarrow\infty}E_{P_{n_k}}[\varphi(Y_1,\cdots,Y_d)]=E_P[\varphi(Y_1,\cdots,Y_d)].
$$
 Denote
$$\text{core}(\Sbep)=\{P: P\text{ is a probability measure on } (\Omega,\sigma(\mathscr{H})), \;\; E_P[f]\le \Sbep [f] \;\; \forall f\in \mathscr{H}_b\},$$
which is called the core of $\Sbep$.
 Zhang \cite{Zhang2021, Zhang2023, Zhang2024} proved the following lemma.
 \begin{lemma}\label{lem:0.1} The following are equivalent:
 \begin{description}
          \item[\rm (i)] Condition (CC) holds for some $\mathscr{P}$;
          \item[\rm (ii)] Condition (CC) holds for $\text{core}(\Sbep)$;
          \item[\rm (iii)] $\Sbep$ is regular on $\mathscr{H}_b$, i.e., whenever $\mathscr{H}_b\ni \varphi_n\searrow 0$, we have $\Sbep [\varphi_n]\searrow 0$.
        \end{description}
        Moreover, condition (CC) holds for $\mathscr{P}$ if and only if $\mathscr{P}$ satisfies (\ref{eqexpressbyP}), and for any $Y_i\in \mathscr{H}$, $i=1,2,\ldots$, $\mathscr{P}\bm Y^{-1}$ is a weakly compact family of probability measures on $\mathbb R^{\mathbb N}$, where $\bm Y=(Y_1,Y_2,\ldots)$.
  \end{lemma}

  It should be pointed out that $\sup_{P\in\mathscr{P}_1}E_P[X] =\sup_{P\in\mathscr{P}_2}E_P[X]$ ($\forall X\in \mathscr{H}_b$) does not imply $\sup_{P\in\mathscr{P}_1}P(A) =\sup_{P\in\mathscr{P}_2}P(A)$ ($\forall A\in\sigma(\mathscr{H})$).

\par In this paper, for real numbers $x$ and $y$, we write $x\vee y=\max\{x,y\}$, $x\wedge y=\min\{x,y\}$, $x^+=x\vee 0$, and $x^-=(-x)\vee 0$. For a random variable $X$, we denote its truncation by $(-c)\vee X\wedge c$ as $X^{(c)}$. If the limit exists, we define $\breve{\mathbb E}[X]=\lim_{c\rightarrow\infty}\Sbep [X^{(c)}]$, and $\breve{\mathcal E}[X]=-\breve{\mathbb E}[-X]$. Clearly, if $C_{\upCapc}(|X|)<\infty$, then $\breve{\mathbb E}[X]$, $\breve{\mathcal E}[X]$, and $\breve{\mathbb E}[|X|]$ are well-defined, and $\breve{\mathbb E}[X], \breve{\mathcal E}[X]\le \breve{\mathbb E}[|X|]\le C_{\upCapc}(|X|)$.

   \section{Law of the Iterated Logarithm}\label{sectLIL}
\setcounter{equation}{0}

 Let $\{Y_n; n\ge 1\}$ be a sequence of independent and identically distributed random variables on the sub-linear expectation space $(\Omega,\mathscr{H},\Sbep)$, with capacity $\Capc$ satisfying (\ref{eq1.3}), and let $S_n=\sum_{i=1}^nY_i$. Denote $\overline{\sigma}^2=\breve{\mathbb E}[Y_1^2]=\lim_{c\to \infty}\Sbep[Y_1^2\wedge c]$, $\underline{\sigma}^2=\breve{\mathcal E}[Y_1^2]=\lim_{c\to \infty}\cSbep[Y_1^2\wedge c]$. $[\underline{\sigma},\overline{\sigma}]$ is called the standard deviation interval of $Y_1$. Zhang \cite{Zhang2021,Zhang2022} obtained the following law of the iterated logarithm.

   \begin{theorem} \label{thLIL1}
Assume
  \begin{align}\label{eqLILmomentcondition1} C_{\Capc}\bigg[\frac{Y_1^2}{\log\log|Y_1|}\bigg]<\infty,
    \end{align}
    \begin{align}\label{eqLILmeanzerocondition}\breve{\mathbb E}[Y_1]=\breve{\mathbb E}[-Y_1]=0, \end{align}
    \begin{align}\label{eqLILmomentcondition2}\overline{\sigma}^2<\infty. \end{align}
 Then
\begin{align}\label{eqthLILiid.3}
\outCapc\bigg( \limsup_{n\to \infty}\frac{|S_n|}{\sqrt{2n\log\log n}}> \overline{\sigma} \bigg)=0,
\end{align}
\begin{align}\label{eqthLILiid.4}\outcCapc\bigg([-\overline{\sigma}, \;\;\overline{\sigma}]\supset Cl\Big\{\frac{S_n}{\sqrt{2n\log\log n}}\Big\} \supset
[-\underline{\sigma}, \;\;\underline{\sigma}]\bigg)=1,
\end{align}
where $Cl \{x_n\}$ denotes the set of limit points (cluster points) of $\{x_n\}$ in $\mathbb R$.

Furthermore, assume (\ref{eqLILmomentcondition1}), (\ref{eqLILmeanzerocondition}), (\ref{eqLILmomentcondition2}), and
   \begin{align}\label{eqLILmomentcondition3} \underline{\sigma}^2>0,
    \end{align}
and condition (CC) is satisfied. Then for $V=\Capc^{\mathscr{P}}$ or $\outCapc$, we have
 \begin{align} \label{eq:LILcompact}
V\bigg( Cl \bigg\{\frac{S_n}{\sqrt{2n\log\log n}}\bigg\}=[-\sigma, \sigma]\bigg)=\begin{cases}
  1, & \text{if } \sigma\in [\underline{\sigma},\overline{\sigma}],\\
  0, & \text{if } \sigma\not\in [\underline{\sigma},\overline{\sigma}].
  \end{cases}
\end{align}

Conversely, if condition (CC) is satisfied and for $V=\Capc^{\mathscr{P}}$ or $\outCapc$,
$$V\bigg( \limsup_{n\to \infty} \frac{|S_n|}{\sqrt{2n\log\log n}}=\infty\bigg)<1, $$
then (\ref{eqLILmomentcondition1})-(\ref{eqLILmomentcondition2}) hold.
\end{theorem}

The purpose of this paper is to establish the limit form of the law of the iterated logarithm as in (\ref{eq:chenLIL}), and to show that $\sigma$ in (\ref{eq:LILcompact}) can be a random variable. The following is our main theorem.
 \begin{theorem}\label{thLIL2} Assume (\ref{eqLILmomentcondition1}), (\ref{eqLILmeanzerocondition}), (\ref{eqLILmomentcondition2}), and condition (CC) holds for the family of probability measures $\mathscr{P}$. Then
 \begin{description}
   \item[\rm (i)] For any continuous function $\sigma(x_1,x_2,\ldots)$ on $\mathbb R^{\mathbb N}$ taking values in $[\underline{\sigma},\overline{\sigma}]$, there exists $P\in\mathscr{P}$ such that
 \begin{align} \label{eq:thLIL2:1}
 P\bigg(Cl\bigg\{\frac{S_n}{\sqrt{2n\log\log n}}\bigg\}=\Big[-\sigma(Y_1,Y_2,\ldots),\sigma(Y_1,Y_2,\ldots)\Big]\bigg)=1,
 \end{align}
 \begin{align}\label{eq:thLIL2:2}
 &P\bigg(\lim_{n\to \infty} \max_{k\le n} \frac{S_k}{\sqrt{2k\log\log n}}=\sigma(Y_1,Y_2,\ldots) \nonumber\\
 & \qquad \text{and } \lim_{n\to \infty} \min_{k\le n} \frac{S_k}{\sqrt{2k\log\log n}}=-\sigma(Y_1,Y_2,\ldots)\bigg)=1;
  \end{align}
   \item[\rm (ii)] For any finite-dimensional Borel-measurable function $\sigma(x_1,\ldots,x_d)$ taking values in $[\underline{\sigma},\overline{\sigma}]$, there exists $P\in\text{core}(\Sbep)$ such that (\ref{eq:thLIL2:1}) and (\ref{eq:thLIL2:2}) hold for $\sigma(x_1,x_2,\ldots)=\sigma(x_1,\ldots,x_d)$.
 \end{description}

 It is easy to see that in (i), replacing $P$ by $\Capc^{\mathscr{P}}$ or $\outCapc$, (\ref{eq:thLIL2:1}) and (\ref{eq:thLIL2:2}) hold; in (ii), replacing $P$ by $\Capc^{\text{core}(\Sbep)}$ or $\outCapc$, (\ref{eq:thLIL2:1}) and (\ref{eq:thLIL2:2}) hold.

 \end{theorem}

  \section{Proof of the Main Results}\label{sectProof}
  \setcounter{equation}{0}

For two sequences $\{f_n;n\ge 1\}$ and $\{g_n\ge 0;n\ge 1\}$, $f_n=o(g_n)$ means $\frac{f_n}{g_n}\to 0$, $f_n=O(g_n)$ means $\limsup_n \frac{|f_n|}{g_n}<\infty$, $f_n\sim g_n$ means $\frac{f_n}{g_n}\to 1$, and $f_n\approx g_n$ means $0<\liminf_n \frac{f_n}{g_n}\le \limsup_n \frac{f_n}{g_n}<\infty$.

Let $p>2$. As in Zhang \cite{Zhang2022}, denote $t_j=\sqrt{2\log\log j}$, $b_j=\alpha_j \sqrt{j}/\sqrt{2\log\log j}$, where $\alpha_j>0$ satisfies $\alpha_j\to 0$, $b_j\nearrow \infty$, and $\alpha_j^{1-p}t_j^{-2}\to 0$. For example, we may take $\alpha_j=1/\sqrt{\log\log\log j}$. Let $V_n^2=\sum_{j=1}^n (Y_j^2\wedge b_j^2)$. Theorem \ref{thLIL2} can be obtained by combining the following two propositions when  (\ref{eqLILmomentcondition3}) is satisfied. 

\begin{proposition} \label{prop:1} Assume that there exists a family of probability measures $\mathscr{P}$ on $(\Omega,\sigma(\mathscr{H}))$ such that the sub-linear expectation $\Sbep$ satisfies (\ref{eqexpressbyP}). Assume conditions (\ref{eqLILmomentcondition1}), (\ref{eqLILmeanzerocondition}), (\ref{eqLILmomentcondition2}), and (\ref{eqLILmomentcondition3}) hold. Then for any $P\in \mathscr{P}$, we have
\begin{align}\label{eq:prop1.1} P\bigg(Cl\bigg\{ \frac{S_n}{V_n\sqrt{2\log\log n}}\bigg\}=[-1,1]\bigg)=1,
\end{align}
\begin{align}\label{eq:prop1.2} P\bigg(\lim_{n\to \infty} \max_{k\le n} \frac{S_k}{V_k\sqrt{2\log\log n}}=1\text{ and } \lim_{n\to \infty} \min_{k\le n} \frac{S_k}{V_k\sqrt{2\log\log n}}=-1\bigg)=1.
\end{align}
 \end{proposition}

 \begin{proposition} \label{prop:2} Assume conditions (\ref{eqLILmomentcondition1}) and (\ref{eqLILmomentcondition2}) hold.
Then
\begin{align}\label{eq:prop2.1}\outcCapc\bigg(\underline{\sigma}^2\le \liminf_{n\to \infty}\frac{V_n^2}{n}\le \limsup_{n\to \infty}\frac{V_n^2}{n}\le \overline{\sigma}^2\bigg)=1.
\end{align}
Furthermore, if condition (CC) holds for the family of probability measures $\mathscr{P}$, then for any continuous function $\varphi(\bm x)$ on $\mathbb R^{\mathbb N}$, there exists $P\in\mathscr{P}$ such that
\begin{align}\label{eq:prop2.2} P\bigg(\lim_{n\to \infty}\frac{V_n^2}{n}=
\underline{\sigma}^2\vee \varphi(Y_1,Y_2,\ldots)\wedge\overline{\sigma}^2 \bigg)=1,
\end{align}
and for any finite-dimensional Borel-measurable function $\varphi(x_1,\ldots,x_d)$, there exists $P\in\text{core}(\Sbep)$ such that (\ref{eq:prop2.2}) holds for $\varphi(x_1,x_2,\ldots)=\varphi(x_1,\ldots,x_d)$.
 \end{proposition}

 \begin{remark} If condition (\ref{eqLILmomentcondition1}) is strengthened to $C_{\Capc}(Y_1^2)<\infty$, then $V_n^2$ in (\ref{eq:prop1.1})-(\ref{eq:prop2.2}) can be replaced by $\sum_{j=1}^n Y_j^2$. In this case, Proposition \ref{prop:2} is actually the strong law of large numbers for independent and identically distributed random variables $\{Y_j^2;j\ge 1\}$, see Zhang \cite{Zhang2024}, and $C_{\Capc}(Y_1^2)<\infty$ is also a necessary condition for the strong law of large numbers. Now $\{Y_j^2\wedge b_j^2;j\ge 1\}$ is a sequence of independent but not identically distributed random variables, and condition (\ref{eqLILmomentcondition1}) is weaker than $C_{\Capc}(Y_1^2)<\infty$. We will follow the approach of Zhang \cite{Zhang2024} in proving the strong law of large numbers to prove Proposition \ref{prop:2}.

 Proposition \ref{prop:1} is a self-normalized law of the iterated logarithm. When $P$ is replaced by $\mathbb V^{\mathcal{P}}$, (\ref{eq:prop1.1}) has been proved in Zhang \cite{Zhang2022}, where a key tool is de la Pe\~{n}a et al. \cite[Lemma 13.8]{PLS2009} on the self-normalized law of the iterated logarithm for martingale difference sequences. The key point of this paper is to prove the limit form of the self-normalized law of the iterated logarithm (\ref{eq:prop1.2}), while (\ref{eq:prop1.1}) is a direct consequence of (\ref{eq:prop1.2}). Since there is no ready-made result on the limit form of the law of the iterated logarithm for martingale difference sequences, we will use the Skorokhod embedding theorem to obtain (\ref{eq:prop1.2}).
 \end{remark}

To prove Propositions \ref{prop:1} and \ref{prop:2}, we need some lemmas. The first one is about exponential inequalities, see Zhang \cite{Zhang2016,Zhang2021,Zhang2022}.
\begin{lemma}\label{lem:ExpIneq} Let $\{X_1,\ldots, X_n\}$ be independent random variables on $(\Omega, \mathscr{H}, \Sbep)$. Denote
$A_n(p,y)=\sum_{i=1}^n\Sbep[(X_i^+\wedge y )^p]$,
$ \breve{B}_{n,y}=\sum_{i=1}^n \breve{\mathbb E}[(X_i\wedge y)^2]$. Then for any $p\ge 2$, $x,y>0$, $0<\delta\le 1$, we have
$$
  \upCapc\Big( \max_{k\le n} \sum_{i=1}^k(X_i-\breve{\mathbb E}[X_i])\ge x\Big)
\le    \upCapc\big(\max_{k\le n} X_k> y \big)
+\exp\bigg\{-\frac{x^2}{2(xy+\breve{B}_{n,y})   }\bigg\},
$$
\begin{align*}
& \upCapc\Big( \max_{k\le n} \sum_{i=1}^k(X_i-\breve{\mathbb E}[X_i])\ge x\Big) \\
\le & \upCapc\big(\max_{k\le n} X_k> y \big)
  +2\exp\{p^p\}\bigg\{\frac{A_n(p,y)}{y^p} \bigg\}^{\frac{\delta x}{10y}}
+\exp\bigg\{-\frac{x^2}{2\breve{B}_{n,y}(1+\delta)   }\bigg\}.
\end{align*}
 \end{lemma}

 The following exponential inequality for martingales can be found in de la Pe\~{n}a \cite{Pena1999}, see Theorem 9.12 of de la Pe\~{n}a et al. \cite{PLS2009}.
\begin{lemma}\label{lem:ExpIneqM} Let $\{X_n; n\ge 1\}$ be a martingale difference sequence on a probability space $(\Omega,\mathcal{F},\pr)$ with respect to the $\sigma$-field filtration $\{\mathcal F_n\}$, and assume $|X_n|\le c$ a.s. Then
$$ \pr\bigg( \sum_{i=1}^n X_i\ge x \text{ and } \sum_{i=1}^n \ep[X_i^2|\mathcal{F}_{i-1}]\le y\; \text{ for some } n \text{ hold} \bigg)
\le \exp\bigg\{ -\frac{x^2}{2(xc+y)}\bigg\}.   $$
\end{lemma}

The following lemma is from Guo and Li \cite{GL21} (see also Hu et al. \cite{HLL2021} and Gao et al. \cite{GLL21}).
\begin{lemma}\label{lem3} Let $\{X_n;n\ge 1\}$ be a sequence of independent random variables on a sub-linear expectation space $(\Omega,\mathscr{H},\Sbep)$ satisfying (\ref{eqexpressbyP}). Denote
$$ \mathcal{F}_n=\sigma(X_1,\ldots,X_n) \; \text{ and }\; \mathcal{F}_0=\{\emptyset,\Omega). $$
Then for each $P\in\mathscr{P}$, we have
$$ E_P\big[\varphi(X_n)|\mathcal{F}_{n-1}\big]\le \Sbep[\varphi(X_n)]\;\; a.s., \;\; \varphi\in C_{b,Lip}(\mathbb R). $$
\end{lemma}

The following lemma is Lemma 6.1 in Zhang \cite{Zhang2021}.

\begin{lemma} \label{lem2} Let $X\in \mathscr{H}$.
\begin{description}
  \item[\rm (i)]
 For any $\delta>0$, we have
$$ \sum_{n=1}^{\infty} \Capc\big(|X|\ge \delta \sqrt{n\log\log n} \big)<\infty \;\; \Longleftrightarrow C_{\Capc}\bigg[\frac{X^2}{\log\log|X|}\bigg]<\infty.
$$
 \item[\rm (ii)]
   If $C_{\Capc}\bigg[\frac{X^2}{\log\log|X|}\bigg]<\infty$, then for any $\delta>0$ and $p>2$, we have
$$ \sum_{n=1}^{\infty} \frac{\Sbep\big[\big(|X|\wedge (\delta \sqrt{n\log\log n})\big)^p\big]}{(n\log\log n)^{p/2}}<\infty. $$
 \item[\rm (iii)] If $C_{\Capc}\bigg[\frac{X^2}{\log\log|X|}\bigg]<\infty$, then for any $\delta>0$, we have
 $$ \Sbep[X^2\wedge (2\delta n\log\log n)]=o(\log \log n), $$
  $$ \breve{\mathbb E}[(|X|-\delta \sqrt{2 n\log\log n})^+]=o(\sqrt{\log\log n/n}). $$
\end{description}
\end{lemma}

\begin{lemma} \label{lem4} Let $X\in \mathscr{H}$, and assume $C_{\Capc}\bigg[\frac{X^2}{\log\log|X|}\bigg]<\infty$. Then for any $\delta>0$, we have
$$ \sum_{n=1}^{\infty} \frac{\Sbep\big[\big(|X|\wedge (\delta \sqrt{n/\log\log n})\big)^4\big]}{n^2}<\infty. $$
\end{lemma}

\begin{proof} Without loss of generality, assume $\delta=1$. Let $B(x)=\sqrt{x/\log\log x}$. Then $\text{\rm inv} B(y)\approx y^2\log \log y$. Note that when $Y$ is bounded, we have $\Sbep[|Y|]=\breve{\mathbb E}[|Y|]\le C_{\Capc}(|Y|)$. We obtain
\begin{align*}
 &\sum_{n={3^3}}^{\infty} \frac{\Sbep\big[\big(|X|\wedge (\delta \sqrt{n/\log\log n})\big)^4\big]}{n^2}
 \le 2\int_{e^e}^{\infty}\frac{\Sbep\big[\big(|X|\wedge B(x)\big)^4\big]}{x^{2}}dx\\
 \le & 2\int_{e^e}^{\infty}\frac{C_{\Capc}\big[\big(|X|\wedge B(x)\big)^4\big]}{x^{2}}dx
 \le 2 \int_{e^e}^{\infty}\int_0^{B^4(x)}\frac{\Capc\big(|X|^4\ge y\big)}{x^2}dy dx\\
 = &8 \int_{e^e}^{\infty}\int_0^{B(x)}y^3\frac{\Capc\big(|X|\ge y\big)}{x^2}dy dx \le C +8 \int_{e^e}^{\infty} y^3\Capc\big(|X|\ge y\big)\int_{x: B(x)\ge y}\frac{1}{x^2}dx dy\\
 =& C +8 \int_{e^e}^{\infty} y^3\Capc\big(|X|\ge y\big)\frac{1}{\text{\rm inv} B(y)} dy\\
 \le & C+C\int_{e^e}^{\infty}\Capc\bigg(\frac{|X|^2}{\log\log |X|}\ge \frac{y^2}{\log\log y}\bigg)\frac{y}{\log \log y}dy\\
 \le
& C+C\int_{0}^{\infty}\Capc\bigg(\frac{|X|^2}{\log\log |X|}\ge x\bigg)dx<\infty.
\end{align*}
\end{proof}

\bigskip

\begin{proof}[\bf Proof of Proposition \ref{prop:1}]
Note that $\{Y_n; n\ge 1\}$ is a sequence of independent and identically distributed random variables on the sub-linear expectation space $(\Omega,\mathscr{H},\Sbep)$, and $\Sbep$ satisfies (\ref{eqexpressbyP}). Let $P\in \mathscr{P}$. Denote $\mathcal{F}_n=\sigma(Y_1,\ldots,Y_n)$, $\mathcal{F}_0=\{\emptyset,\Omega\}$, $Z_n=Y_n^{(b_n)}$, $\Delta M_n=Z_n-E_P[Z_n|\mathcal{F}_{n-1}]$, and $M_n=\sum_{i=1}^n \Delta M_i$. Then $\{M_n,\mathcal{F}_n; n\ge 1\}$ is a martingale under the probability measure $P$. By Lemma \ref{lem3}, we have
$$ E_P[Z_j|\mathcal{F}_{j-1}]\le \Sbep[Z_j]=\Sbep[ Y_1^{( b_j)}]\to\breve{\mathbb E}[Y_1]=0. $$
Similarly,
$$ E_P[-Z_j|\mathcal{F}_{j-1}]\le \Sbep[-Z_j]=\Sbep[-Y_1^{( b_j)}]\to\breve{\mathbb E}[-Y_1]=0. $$
Therefore,
$$ \big|E_P[Z_j|\mathcal{F}_{j-1}]\big|\to 0 \; a.s.\; P. $$
On the other hand, by (\ref{eq:prop2.1}) (which has been proved in Zhang \cite{Zhang2022}), (\ref{eqLILmomentcondition2}), and (\ref{eqLILmomentcondition3}), we have, under the probability measure $P$,
\begin{align}\label{eq:prop:1.1} V_n^2=\sum_{j=1}^n Z_j^2\approx n\;\; a.s.,
\end{align}
and hence
\begin{align}\label{eq:prop:1.2} U_n^2=:\sum_{j=1}^n (\Delta M_j)^2=\sum_{j=1}^n (Z_j-E_P[Z_j|\mathcal{F}_{j-1}])^2\sim V_n^2 \approx n\;\; a.s.
\end{align}
Furthermore, by Lemma \ref{lem3}, we have
\begin{align}\label{eq:prop:1.2ad}
 E_P\big[|\Delta M_j|^p|M_1,\ldots, M_{j-1}\big]= E_P\big[|\Delta M_j|^p|\mathcal{F}_{j-1}\big]
\le C_p E_P\big[|Z_j|^p|\mathcal{F}_{j-1}\big]
\le C_p \Sbep\big[|Z_j|^p\big]\; a.s.
\end{align}
Next we prove
\begin{align}\label{eq:LILforZ1} & P\bigg(\lim_{n\to\infty}\max_{k\le n} \frac{\sum_{j=1}^k (Z_j-E_P[Z_j|\mathcal{F}_{j-1}])}{V_k (2\log\log n)^{1/2}}=1\bigg)=1.
\end{align}
Note that by (\ref{eq:prop:1.2}), it suffices to prove
\begin{align}\label{eq:LILforZ1-2} & P\bigg(\lim_{n\to\infty}\max_{k\le n} \frac{M_k}{U_k (2\log\log n)^{1/2}}=1\bigg)=1.
\end{align}
By the Skorokhod embedding theorem (see Theorem A.1 of Hall and Heyde \cite{HH1980}), without loss of generality, we may assume
$M_n=W(T_n)$, $T_n=\sum_{j=1}^n \tau_j$,
where $\{W(t);t\ge 0\}$ is a standard Brownian motion, $\tau_j$ are $\mathscr{G}_j$-measurable nonnegative random variables, $\mathscr{G}_j=\sigma(M_1,\ldots,M_j,W(t),0\le t\le T_j)$,
and
\begin{align}\label{eq:prop:1.3}
\begin{aligned} & E_P[\tau_j|\mathscr{G}_{j-1}]=E_P\big[(\Delta M_j)^2|M_1,\ldots, M_{j-1}\big], \\
 & E_P[\tau_j^r|\mathscr{G}_{j-1}]\le C_r E_P\big[|\Delta M_j|^{2r}|M_1,\ldots, M_{j-1}\big], \; r\ge 1.
 \end{aligned}
\end{align}
By the limit form of the law of the iterated logarithm for Brownian motion (see Chen \cite{Chen2015}), we have
\begin{align}\label{eq:embed:2} \lim_{n\to \infty}\max_{1\le s\le T_n}\frac{W(s)}{\sqrt{s}\sqrt{2\log \log T_n}}=1\;\; a.s. \text{ under } P.
\end{align}

Note that $\{\tau_j-E_P[\tau_j|\mathscr{G}_{j-1}]; j\ge 1\}$ is a martingale difference sequence under the probability measure $P$, and by (\ref{eq:prop:1.3}), (\ref{eq:prop:1.2ad}), and Lemma \ref{lem4}, we have
\begin{align}\label{eq:prop:1.4}
&\sum_{j=3}^{\infty}\frac{ E_P\big[(\tau_j-E_P[\tau_j|\mathscr{G}_{j-1}])^2\big]}{j^2}
\le \sum_{j=3}^{\infty}\frac{ E_P\big[\tau_j^2\big]}{j^2}
\le C\sum_{j=3}^{\infty}\frac{ E_P\big[(\Delta M_j )^4\big]}{j^2} \nonumber\\
\le &
      C\sum_{j=3}^{\infty}\frac{ \Sbep\big[Z_j^4\big]}{j^2}
\le C\sum_{j=3}^{\infty} \frac{\Sbep\big[\big(|Y_1|\wedge (\delta \sqrt{j/\log\log j})\big)^4\big]}{j^2}<\infty.
\end{align}
Thus, by the law of large numbers for martingale differences, we have
\begin{align}\label{eq:prop:1.5} \frac{\sum_{j=1}^n(\tau_j-E_P[\tau_j|\mathscr{G}_{j-1}])}{n}\to 0 \;\; a.s. \text{ under } P.
\end{align}
Similarly, for the martingale difference sequence $\{(\Delta M_j)^2-E_P[(\Delta M_j)^2|M_1,\ldots, M_{j-1}]; j\ge 1\}$, we also have
\begin{align}\label{eq:prop:1.6} \frac{\sum_{j=1}^n \big( (\Delta M_j)^2-E_P[(\Delta M_j)^2|M_1,\ldots, M_{j-1}]\big)}{n}\to 0\;\; a.s. \text{ under } P.
\end{align}
Combining (\ref{eq:prop:1.2}), (\ref{eq:prop:1.3}), (\ref{eq:prop:1.5}), and (\ref{eq:prop:1.6}), we obtain
\begin{align}\label{eq:prop:1.7}
 &\sum_{j=1}^n \tau_j \sim \sum_{j=1}^n E_P[\tau_j|\mathscr{G}_{j-1}] = \sum_{j=1}^n E_P[(\Delta M_j)^2|M_1,\ldots, M_{j-1}] \sim U_n^2\approx n \;\; a.s. \text{ under } P.
\end{align}
At the same time, (\ref{eq:prop:1.4}) implies that $\tau_n/n\to 0$ a.s. under $P$. Thus, by the fact that
\begin{align}\label{eq:prop:1.7ad}
f_n \ge 0, \;\; 0<g_n\nearrow \infty \text{ and } \frac{f_n}{g_n}\to 0 \implies \frac{\max_{k\le n}f_k}{g_n}\to 0,
\end{align}
we have
$$ \frac{\max_{k\le n}\tau_k}{n}\to 0 \;\; a.s. \text{ and } \frac{\max_{k\le n}\tau_k}{T_n}\to 0 \;\; a.s. \text{ under } P. $$
By the properties of Brownian motion (see Theorem 1.2.1 of Cs\H org\"o and R\'ev\'esz \cite{CR1981}), we have
$$ \frac{\max_{T_{k-1}\le s\le T_k}|W(s)-W(T_k)|}{\sqrt{T_k\log\log T_k}}\to 0 \;\; a.s. \text{ under } P. $$
From this we obtain
\begin{align*}
&\max_{T_{k-1}\le s\le T_k}\bigg|\frac{W(s)}{\sqrt{s}}-\frac{W(T_k)}{\sqrt{T_k}}\bigg|\\
\le &\frac{\max_{T_{k-1}\le s\le T_k}|W(s)-W(T_k)|}{\sqrt{T_{k}}}
+\frac{|W(T_k)|}{\sqrt{T_k}}\bigg(\sqrt{\frac{T_k}{T_{k-1}}}-1\bigg)\\
=& o\bigg( \sqrt{2\log\log T_k}\bigg)+O\bigg( \sqrt{2\log\log T_k}\bigg)\cdot o(1)=o\Big( \sqrt{2\log\log T_k}\Big)\;\; a.s. \text{ under } P.
\end{align*}
Thus,
\begin{align*}
& \bigg|\max_{1\le s\le T_n}\frac{W(s)}{\sqrt{s}}-\max_{1\le k\le n}\frac{W(T_k)}{\sqrt{T_k}}\bigg|\\
\le & \max_{2\le k\le n} \max_{T_{k-1}\le s\le T_k}\bigg|\frac{W(s)}{\sqrt{s}}-\frac{W(T_k)}{\sqrt{T_k}}\bigg|+O(1)
= o\Big( \sqrt{2\log\log T_n}\Big)\;\; a.s. \text{ under } P.
\end{align*}
Combining the above with (\ref{eq:embed:2}), we obtain
\begin{align}\label{eq:prop:1.8}
& \lim_{n\to \infty} \max_{k\le n}\frac{M_k}{\sqrt{T_k}\sqrt{2\log\log n}}
= \lim_{n\to \infty} \max_{k\le n}\frac{W(T_k)}{\sqrt{T_k}\sqrt{2\log\log T_n}}=1 \;\; a.s. \text{ under } P.
\end{align}
Therefore, by (\ref{eq:prop:1.7}) and (\ref{eq:prop:1.8}), we obtain (\ref{eq:LILforZ1-2}). Thus (\ref{eq:LILforZ1}) is proved.

Next we prove that $Z_j-E_P[Z_j|\mathcal{F}_{j-1}]$ in (\ref{eq:LILforZ1}) can be replaced by $Y_j$.
Let $\lambda > 1$. As in Zhang \cite{Zhang2022}, denote $d_n=\sqrt{2n\log\log n}$, $n_k=[\lambda^k]$, $I(k)=\{n_k+1,\ldots,n_{k+1}\}$. Then
$n_k/n_{n_{k+1}}\to {1}/{\lambda}$, $d_{n_k}/d_{n_{k+1}}\to 1/\sqrt{\lambda}$.

For $p>2$, by Lemma \ref{lem2} (ii), we have
\begin{align}\label{eq:prop:1.15} \sum_{k=1}^{\infty} \frac{ \Lambda_{n_k,n_{k+1}}(p)}{d_{n_{k+1}}^p}<\infty,
\end{align}
where
$$\Lambda_{n_k,n_{k+1}}(p)=\sum_{j\in I(k)}\Sbep[\big((|Y_j|\wedge d_{n_{k+1}}\big)^p]. $$
Let
$$ \mathbb N_1=\bigg\{k\in \mathbb N; \frac{ \Lambda_{n_k, n_{k+1}}(p)}{d_{n_{k+1}}^p}\le t_{n_{k+1}}^{-2p}\bigg\}.
$$
Since $\alpha_j\to 0$, $\alpha_j^{1-p}t_j^{-2}\to 0$, when $k\in \mathbb N_1$, we have
\begin{align}\label{eq:prop:1.16}
 &\frac{\sum_{j\in I(k)}
\Sbep[ |Y_j^{(d_{n_{k+1}})}-Z_j|]}{d_{n_{k+1}}}
\le C\frac{\Lambda_{n_k,n_{k+1}}(p)}{d_{n_{k+1}}^p} \alpha_{n_{k+1}}^{1-p}t_{n_{k+1}}^{2p-2} \le C\alpha_{n_{k+1}}^{1-p}t_{n_{k+1}}^{-2}\to 0, \\
 & \label{eq:prop:1.17}
 \frac{\sum_{j\in I(k)}
  \Sbep[(Y_j^{(d_{n_{k+1}})}-Z_j)^2] }{ n_{k+1} }
   \le C\frac{\Lambda_{n_k,n_{k+1}}(p)}{d_{n_{k+1}}^p} \alpha_{n_{k+1}}^{2-p}t_{n_{k+1}}^{2p-2} \le C \alpha_{n_{k+1}}^{2-p}t_{n_{k+1}}^{-2}\to 0,
\end{align}
Define
$$ Z_{j,1}=Z_j \text{ if } j\in I(k)\; \text{ and }\; k\in \mathbb N_1, \text{ otherwise } 0, $$
$$ Z_{j,2}=0 \text{ if } j\in I(k)\; \text{ and }\; k\in \mathbb N_1, \text{ otherwise } Z_j. $$
Then $Z_j=Z_{j,1}+Z_{j,2}$. The definitions of $d_n$, $n_k$, $\alpha_j$, $t_j$, $\Lambda_{n_k,n_{k+1}}(p)$, $\mathbb N_1$, $Z_{j,1}$, $Z_{j,2}$, etc. are the same as in Zhang \cite{Zhang2022}. Using (\ref{eq:prop:1.15})-(\ref{eq:prop:1.17}) together with Lemmas \ref{lem:ExpIneq} and \ref{lem:ExpIneqM}, Zhang \cite{Zhang2022} proved
$$ \frac{ \sum_{j=1}^n (Z_{j,2}-E_P[Z_{j,2}|\mathcal{F}_{j-1}]) }{\sqrt{2n\log\log n}}\to 0\;\; a.s. \text{ under } P, $$
$$ \outcCapc\bigg( \frac{\sum_{j=1}^n (Y_j-Z_{j,1}) }{ \sqrt{2n\log\log n}}\to 0\bigg)=1, $$
 $$ \frac{\sum_{j=1}^n |E_P[Z_{j,1} |\mathcal{F}_{j-1}]|}{\sqrt{2n\log\log n}}\to 0\;\; a.s. \text{ under } P. $$
See (3.9) and the preceding equation, (3.15) and the preceding equation, and (3.10) and the preceding equation in Zhang \cite{Zhang2022}. Therefore,   
\begin{align}\label{eq:remainder}
& \lim_{n\to \infty} \frac{\sum_{j=1}^n \big(Y_j-(Z_j-E_P[Z_j|\mathcal{F}_{j-1}])\big) }{ V_n \sqrt{2\log\log n}} \nonumber \\
= & \lim_{n\to \infty} \frac{\sum_{j=1}^n \big(Y_j-(Z_j-E_P[Z_j|\mathcal{F}_{j-1}])\big) }{ \sqrt{2n\log\log n}}=0 \;\; a.s. \text{ under } P.
\end{align}
From this and the fact (\ref{eq:prop:1.7ad}), we obtain
$$ \lim_{n\to \infty} \max_{k\le n}\frac{|\sum_{j=1}^k \big(Y_j-(Z_j-E_P[Z_j|\mathcal{F}_{j-1}])\big)| }{ V_k \sqrt{2\log\log n}} =0 \;\; a.s. \text{ under } P. $$
Thus, combining with (\ref{eq:LILforZ1}), we obtain
$$ P\bigg(\lim_{n\to\infty}\max_{k\le n} \frac{\sum_{j=1}^k Y_j}{V_k (2\log\log n)^{1/2}}=1\bigg)=1.
$$
Similarly,
\begin{align*}
& P\bigg(\lim_{n\to\infty}\min_{k\le n} \frac{\sum_{j=1}^k Y_j}{V_k (2\log\log n)^{1/2}}=-1\bigg)
 =P\bigg(\lim_{n\to\infty}\max_{k\le n} \frac{\sum_{j=1}^k (- Y_j)}{V_k (2\log\log n)^{1/2}}=1\bigg)
=1.
\end{align*}
Thus (\ref{eq:prop1.2}) is proved.

It is easy to verify that (\ref{eq:prop1.2}) implies
\begin{align}\label{eq:proof:compact:1} P\bigg(\liminf_{n\to\infty} \frac{S_n}{V_n (2\log\log n)^{1/2}}=-1 \text{ and } \limsup_{n\to\infty} \frac{S_n}{V_n (2\log\log n)^{1/2}}=1 \bigg)=1.
\end{align}
On the other hand, by Lemma \ref{lem2} (i), we have
$$ \sum_{n=1}^{\infty} \Capc(|Y_n|\ge \epsilon\sqrt{2n\log\log n})<\infty, \;\; \forall \epsilon>0. $$
From this and the countable sub-additivity of $\outCapc$, we obtain
$$ \outCapc\bigg( \limsup_{n\to \infty} \frac{|Y_n| }{ \sqrt{2n\log\log n}}>0\bigg)=0. $$
Thus
\begin{align}\label{eq:proof:compact:2}
& P\bigg(\lim_{n\to \infty} \frac{\max_{i\le n}|Y_i| }{ V_n(2\log\log n)^{1/2}}=0\bigg)= P\bigg(\lim_{n\to \infty} \frac{\max_{i\le n}|Y_i| }{ \sqrt{2n\log\log n}}=0\bigg)\nonumber \\
=&P\bigg(\lim_{n\to \infty} \frac{|Y_n| }{ \sqrt{2n\log\log n}}=0\bigg)
\ge \outcCapc\bigg( \lim_{n\to \infty} \frac{|Y_n| }{ \sqrt{2n\log\log n}}=0\bigg)=1.
\end{align}
Furthermore, from (\ref{eq:prop:1.4}),
 $$ \sum_{j=3}^{\infty}\frac{ E_P\big[Z_j^4\big]}{j^2}\le
      \sum_{j=3}^{\infty}\frac{ \Sbep\big[Z_j^4\big]}{j^2} <\infty.
$$
This implies
$ P\Big(\frac{Z_n^2}{n}\to 0\Big)=1. $
Thus under $P$,
\begin{align*}   \frac{V_{n+1}^2}{V_n^2}-1=\frac{Z_{n+1}^2}{V_n^2}\approx \frac{Z_{n+1}^2}{n+1}\to 0\;\; a.s.
\end{align*}
From (\ref{eq:proof:compact:1}), (\ref{eq:proof:compact:2}), and the above inequality, we obtain, under $P$,
\begin{align}\label{eq:proof:compact:4}
&\frac{S_{n+1}}{V_{n+1} (2\log\log (n+1))^{1/2}}-\frac{S_n}{V_n (2\log\log n)^{1/2}}\nonumber\\
=&\frac{Y_{n+1}}{V_{n+1} (2\log\log (n+1))^{1/2}}+ \frac{S_n}{V_n (2\log\log n)^{1/2}}
\bigg(1-\frac{V_n\sqrt{\log\log n}}{V_{n+1}\sqrt{\log\log (n+1)}}\bigg)\nonumber\\
&\to 0\;\;a.s.
\end{align}
It is easy to verify that (\ref{eq:proof:compact:1}) and (\ref{eq:proof:compact:4}) imply (\ref{eq:prop1.1}), see Petrov \cite[Page 250]{Petrov1995}. The theorem is proved.
 \end{proof}

 \bigskip
\begin{proof}[\bf Proof of Proposition \ref{prop:2}] (\ref{eq:prop2.1}) was proved in Zhang \cite{Zhang2022}. Similar to the proof of Theorem 2.1 in Zhang \cite{Zhang2024}, we prove (\ref{eq:prop2.2}) in five steps. Here we assume condition (CC) holds for the family of probability measures $\mathscr{P}$. Thus by Lemma \ref{lem:0.1}, it also holds for $\text{core}(\Sbep)$, and $\mathscr{P}\bm Y^{-1}$ and $\text{core}(\Sbep)\bm Y^{-1}$ are both weakly compact families of probability measures on $\mathbb R^{\mathbb N}$.

{\em Step 1}. We first prove that there exist $\epsilon_k\searrow 0$ and a constant $M>0$ such that
\begin{align}\label{eqconvergence}
\sum_{k=1}^{\infty} \sup_{Q\in \text{core}(\Sbep)} Q\bigg(\max_{ n\le n_{k+1}}\bigg|\sum_{j=1}^n(Z_j^2-E_Q[Z_j^2|\mathcal F_{j-1}])\bigg|\ge \epsilon_k n_k\bigg)\le M<\infty,
\end{align}
where $n_k=2^k$.

Note that $\big|Z_j^2-E_Q[Z_j^2|\mathcal F_{j-1}]\big|\le b_j^2$,
$$E_Q\big[(Z_j^2-E_Q[Z_j^2|\mathcal F_{j-1}])^2|\mathcal{F}_{j-1}\big]
\le E_Q\big[Z_j^4||\mathcal{F}_{j-1}\big]\le \Sbep[Z_j^4]\le b_{j}^2\overline{\sigma}^2.$$
Let $\epsilon_k=\sqrt{\alpha_{n_{k+1}}}$. Applying Lemma \ref{lem:ExpIneqM} to the martingale difference sequence $\{Z_j^2-E_Q[Z_j^2|\mathcal F_{j-1}];j=1,\ldots, n_{k+1}\}$ with $c= b_{n_{k+1}}^2$, $y=n_{k+1}b_{n_{k+1}}^2\overline{\sigma}^2$, and $x=\epsilon_k n_k$, for sufficiently large $k$ we have
\begin{align*}
&\sup_{Q\in \text{core}(\Sbep)}Q\bigg(\max_{ n\le n_{k+1}}\bigg|\sum_{j=1}^n(Z_j^2-E_Q[Z_j^2|\mathcal F_{j-1}])\bigg|\ge \epsilon_k n_k\bigg)\\
\le & 2 \exp\bigg\{-\frac{\epsilon_k^2 n_k^2}{2(\epsilon_k n_{k}b_{n_{k+1}}^2+n_{k+1} b_{n_{k+1}}^2 \overline{\sigma}^2)}\bigg\}\le \exp\{-2\log\log n_k\}.
\end{align*}
Thus (\ref{eqconvergence}) is proved.

{\em Step 2}. Let
\begin{align} \label{eq:core-real:1} \widetilde{\mathscr{P}}_{\max}= \bigg\{\widetilde{P}: \widetilde{P} \text{ is a probability measure on } \mathbb R^{\mathbb N} \text{ satisfying }
    E_{\widetilde{P}}[\varphi]\le \Sbep[\varphi\circ \bm Y], \forall \varphi\in \widetilde{\mathscr{H}}_b\bigg\},
 \end{align}
 where $\bm Y=(Y_1,Y_2,\ldots)$,
\begin{align} \label{eq:core-real:2}\widetilde{\mathscr{H}}_b=\big\{\varphi: \varphi\in C_{b,Lip}(\mathbb R^d)\text{ for some } d\ge 1 \big\}.
\end{align}
We prove
\begin{align}\label{eq:identity}
 \widetilde{\mathscr{P}}_{\max}=\text{core}(\Sbep) \bm Y^{-1}.
\end{align}

This conclusion was proved in Zhang \cite{Zhang2024} (see Lemma 3.4 of Zhang \cite{Zhang2024}). For the sake of completeness, and noting that both $\widetilde{\mathscr{P}}_{\max}$ and $\text{core}(\Sbep)$ are convex sets, we give another proof using the convex separation theorem. To this end, define
$$\widetilde{\mathbb E}[\varphi]=\Sbep[\varphi\circ\bm Y], \; \varphi\in \widetilde{\mathscr{H}}_b. $$
Then $(\mathbb R^{\mathbb N}, \widetilde{\mathscr{H}}_b, \widetilde{\mathbb E})$ is a sub-linear expectation space, and $\widetilde{\mathscr{P}}_{\max}=\text{core}(\widetilde{\mathbb E})$.
By the definition of $\widetilde{\mathscr{P}}_{\max}$, it is clear that $\mathscr{P}\bm Y^{-1}\subset \text{core}(\Sbep)\bm Y^{-1}\subset \widetilde{\mathscr{P}}_{\max}$. Hence
\begin{align} \label{eq:core-real:3}
 \sup_{\widetilde{P}\in \widetilde{\mathscr{P}}_{\max}}E_{\widetilde{P}}[\varphi]=&
\sup_{\widetilde{P}\in \text{core}(\Sbep)\bm Y^{-1}}E_{\widetilde{P}}[\varphi]=\sup_{\widetilde{P}\in \mathscr{P}\bm Y^{-1}}E_{\widetilde{P}}[\varphi]\nonumber\\
=& \Sbep\big[\varphi(Y_1,\ldots, Y_d)\big]=\widetilde{\mathbb E}[\varphi], \; \forall \varphi\in C_{b,Lip}(\mathbb R^d),
\end{align}
On the other hand, since $\mathscr{P}\bm Y^{-1}$ and $\text{core}(\Sbep)\bm Y^{-1}$ are weakly compact families of probability measures on the countable-dimensional space $\mathbb R^{\mathbb N}$, by Lemma \ref{lem:0.1}, $\widetilde{\mathscr{P}}_{\max}=\text{core}(\widetilde{\mathbb E})$ is also a weakly compact family of probability measures on $\mathbb R^{\mathbb N}$.

For any $\varphi\in C_b(\mathbb R^{\mathbb N})$, assume $|\varphi(\bm x)|\le M$. Note that $\widetilde{\mathscr{P}}_{\max}$, $\mathscr{P}\bm Y^{-1}$, and $\text{core}(\Sbep)\bm Y^{-1}$ are all weakly compact, hence tight. For any $\epsilon>0$, there exists a compact set $K\subset \mathbb R^{\mathbb N}$ such that
$$ P(K^c)<\epsilon/(8M), \; P\in \widetilde{\mathscr{P}}_{\max} \bigcup\text{core}(\Sbep)\bm Y^{-1} \bigcup \mathscr{P}\bm Y^{-1}. $$
On the compact set $K$, the continuous function $\varphi$ can be uniformly approximated by bounded Lipschitz functions on finite-dimensional spaces. Thus there exists $\varphi_d\in C_{b,Lip}(\mathbb R^d)$ such that $|\varphi_d(\bm x)|\le 2M$, $\sup_{\bm x\in K}|\varphi_d(\bm x)-\varphi(\bm x)|<\epsilon/2$. Then
$$|\varphi_d(\bm x)-\varphi(\bm x)|< \epsilon/2+4MI\{\bm x\in K^c\}. $$
From this and (\ref{eq:core-real:3}), we obtain
$$\Big|\sup_{\widetilde{P}\in \widetilde{\mathscr{P}}_{\max}}E_{\widetilde{P}}[\varphi]-
\sup_{\widetilde{P}\in \mathscr{P}\bm Y^{-1} } E_{\widetilde{P}}[\varphi] \Big|
<\epsilon+\epsilon=2\epsilon. $$
Hence
\begin{align} \label{eq:core-real:3ad}
 \sup_{\widetilde{P}\in \widetilde{\mathscr{P}}_{\max}}E_{\widetilde{P}}[\varphi]=
\sup_{\widetilde{P}\in \text{core}(\Sbep)\bm Y^{-1}}E_{\widetilde{P}}[\varphi]=
\sup_{\widetilde{P}\in \mathscr{P}\bm Y^{-1}}E_{\widetilde{P}}[\varphi], \; \forall \varphi\in C_{b}(\mathbb R^{\mathbb N}).
\end{align}

To prove (\ref{eq:identity}), it suffices to show that
 $\widetilde{\mathscr{P}}_{\max} \subset \text{core}(\Sbep) \bm Y^{-1}$. Suppose, to the contrary, that $P_0\in \widetilde{\mathscr{P}}_{\max}$, but $P_0\not\in \text{core}(\Sbep)\bm Y^{-1}$. We use the convex separation theorem to derive a contradiction. Denote
 $$\mathbb X=\mathscr{P}(\mathbb R^{\mathbb N})=\big\{\text{finite signed measures on }\mathbb R^{\mathbb N}\big\}. $$
 Endow $\mathbb X$ with the weak convergence topology, equivalently, the family of seminorms $\big(\rho_f\big)_{f\in C_b(\mathbb R^{\mathbb N})}$:
 $$ \rho_f(\mu)=\Big|\int f d\mu\Big|, \mu\in \mathbb X. $$
 Then $\mathbb X$ is a locally convex Hausdorff space, and its dual space $\mathbb X^{\prime}=\{ \text{continuous linear functionals on } \mathbb X\}$ is precisely $C_b(\mathbb R^{\mathbb N})$, i.e., $\mathcal{L}\in \mathbb X^{\prime}$ if and only if there exists $f\in C_b(\mathbb R^{\mathbb N})$ such that
 $$ \mathcal{L}(\mu)=\int fd \mu, \;\; \mu\in \mathbb X. $$
 On the other hand, $\mathbb A=\text{core}(\Sbep)\bm Y^{-1}$ and $\mathbb B=\{P_0\}$ are both compact convex sets in the locally convex Hausdorff space $\mathbb X$, and $\mathbb A \cap \mathbb B=\emptyset$. By the strong separation theorem for convex sets, there exist $\alpha\in\mathbb R$ and a continuous linear functional $\mathcal{L}\in \mathbb X^{\prime}$ such that
 $$ \sup_{\mu\in \mathbb A}\mathcal{L}(\mu)<\alpha<\inf_{\mu \in \mathbb B}\mathcal{L}(\mu). $$
 That is, there exists $f\in C_b(\mathbb R^{\mathbb N})$ such that
 $$ \sup_{\widetilde{P}\in \text{core}(\Sbep)\bm Y^{-1}}E_{\widetilde{P}}[f]<\alpha< E_{P_0}[f]. $$

 But since $P_0\in \widetilde{\mathscr{P}}_{\max}$, by (\ref{eq:core-real:3ad}) we have
 $$ E_{P_0}[f]\le \sup_{\widetilde{P}\in \widetilde{\mathscr{P}}_{\max}}E_{\widetilde{P}}[f]=
\sup_{\widetilde{P}\in \text{core}(\Sbep)\bm Y^{-1}}E_{\widetilde{P}}[f]. $$
This is a contradiction, and (\ref{eq:identity}) is proved.

{\em Step 3}. We prove that for any Borel-measurable function $\varphi_j(x_1,\ldots,x_j), j\ge 1$ on $\mathbb R^j$, there exists a probability measure $Q\in \text{core}(\Sbep)$ such that
\begin{align}\label{eqconditionE} E_Q[Z_j^2|\mathcal F_{j-1}]= \cSbep[Z_j^2]\vee \varphi_{j-1}(Y_1,\ldots,Y_{j-1})\wedge \Sbep[Z_j^2]\;\; a.s. \text{ under } Q, j\ge 2.
\end{align}

Denote $\psi_{i-1}(x_1,\ldots, x_{i-1})=\cSbep[Z_j^2]\vee \varphi_{i-1}(x_1,\ldots,x_{j-1})\wedge \Sbep[Z_j^2]$.
We first construct $\widetilde{Q}\in \widetilde{\mathscr{P}}_{\max}$ such that
\begin{align}\label{eqconditionE2} E_{\widetilde{Q}}[x_d^2\wedge b_d^2 |x_1,\ldots, x_{d-1}]=\psi_{d-1}(x_1,\ldots,x_{d-1}).
\end{align}

To this end, denote $\underline{\mu}_d=\cSbep[Z_d^2]$, $\overline{\mu}_d=\Sbep[Z_d^2]$. Note that $\widetilde{\mathscr{P}}_{\max}$, $\mathscr{P}\bm Y^{-1}$, and $\text{core}(\Sbep)\bm Y^{-1}$ are all weakly compact, and the suprema in (\ref{eq:core-real:3ad}) are attainable. Thus
for any $\varphi\in C_{b,Lip}(\mathbb R^d)$, there exists $\widetilde{P}\in \mathscr{P}\bm Y^{-1}\subset \text{core}(\Sbep)\bm Y^{-1}\subset \widetilde{\mathscr{P}}_{\max}$ such that
\begin{align} \label{eq:core-real:4} E_{\widetilde{P}}[\varphi]=
\sup_{P\in \mathscr{P}}E_P\big[\varphi(Y_1,\ldots, Y_d)\big]=
\Sbep\big[\varphi(Y_1,\ldots, Y_d)\big].
\end{align}
It should be pointed out that when $\varphi$ is merely Borel-measurable, the supremum in (\ref{eq:core-real:3ad}) may not be attainable.

By (\ref{eq:core-real:4}), there exist probability measures $\overline{P}_d$ and $\underline{P}_d$ on $\mathbb R$ such that
\begin{align}\label{eqproofth2.11} E_{\overline{P}_d}[\varphi]\le \Sbep[\varphi(Y_d)], \;\; E_{\underline{P}_d}[\varphi]\le \Sbep[\varphi(Y_d)], \;\; \varphi\in C_{b,Lip}(\mathbb R),
\end{align}
\begin{align}\label{eqproofth2.12} E_{\overline{P}_d}[x^2\wedge b_d^2]=\Sbep[Z_d^2]=\overline{\mu}_d, \;\; E_{\underline{P}_d}[x^2\wedge b_d^2]=\cSbep[Z_d^2]=\underline{\mu}_d.
\end{align}
For a function $\psi_{i-1}(x_1,\ldots,x_{i-1})\in [\underline{\mu}_i,\overline{\mu}_i]$, let
$$ \alpha_{i-1}=\begin{cases} \frac{\psi_{i-1}(x_1,\ldots,x_{i-1})-\underline{\mu}_i}{\overline{\mu}_i-\underline{\mu}_i}, & \text{if } \overline{\mu}_i\ne \underline{\mu}_i,\\
=1, & \text{otherwise}.
\end{cases}
$$
Then $0\le \alpha_{i-1}\le 1$. Define
$$ k_{i-1}(x_1,\ldots,x_{i-1};A)=\alpha_{i-1}\overline{P}_i(A)+(1-\alpha_{i-1})\underline{P}_i(A), \;\; A\in \mathscr{B}^1,$$
$$\widetilde{Q}_1 \text{ is a probability measure on } \mathbb R \text{ satisfying } E_{\widetilde{Q}_1}[\varphi]\le \Sbep[\varphi(Y_1)], \varphi\in C_{b,Lip}(\mathbb R), $$
$$\widetilde{Q}_i(A)=\idotsint_{(x_1,\ldots,x_i)\in A} k_{i-1}(x_1,\ldots,x_{i-1}, d x_i)\widetilde{Q}_{i-1}(dx_1,\ldots,d x_{i-1}), \;\; A\in \mathscr{B}^i, i\ge 2,$$
where $\mathscr{B}^i$ is the Borel $\sigma$-field on $\mathbb R^i$.
Then $\{\widetilde{Q}_i; i\ge 1\}$ is a sequence of probability measures satisfying the consistency condition:
$$ \widetilde{Q}_i(A\times \mathbb R)=\widetilde{Q}_{i-1}(A), \;\; A\in \mathscr{B}^{i-1}. $$
Therefore, by the Kolmogorov extension theorem, there exists a probability measure $\widetilde{Q}$ on $\mathbb R^{\mathbb N}$ such that
$$ \widetilde{Q}\pi_i^{-1}=\widetilde{Q}_i, $$
where $\pi_i(x_1,x_2,\ldots)=(x_1,\ldots,x_i)$ is the projection map.
Now, for $f(x_1,\ldots,x_d)\in C_{b,Lip}(\mathbb R^d)$, by (\ref{eqproofth2.11}) we have
\begin{align*}
& f_{d-1}(x_1,\ldots,x_{d-1})=: \int f(x_1,\ldots,x_d)k_{d-1}(x_1,\ldots,x_{d-1}; dx_d)\\
=& \alpha_{i-1}\int f(x_1,\ldots,x_d)\overline{P}_d(dx_d)+(1-\alpha_{i-1})\int f(x_1,\ldots,x_d)\underline{P}_d(dx_d)\\
\le &\Sbep[f(x_1,\ldots,x_{d-1}, Y_d)].
\end{align*}
It is easy to verify that $\Sbep[f(x_1,\ldots,x_{d-1}, Y_d)]\in C_{b,Lip}(\mathbb R^{d-1})$. By induction, we obtain
\begin{align*}
&E_{\widetilde{Q}}[f(x_1,\ldots,x_d)]=E_{\widetilde{Q}_d}[f(x_1,\ldots,x_d)]=E_{\widetilde{Q}_{d-1}}[f_{d-1}(x_1,\ldots,x_{d-1})]\\
\le & E_{\widetilde{Q}_{d-1}}\big[\Sbep[f(x_1,\ldots,x_{d-1}, Y_d)]\big]\le \Sbep\big[\Sbep[f(x_1,\ldots,x_{d-1}, Y_d)]\big|_{x_1=Y_1,\ldots, x_{d-1}=Y_{d-1}}\big]\\
=& \Sbep[f(X_1,\ldots, Y_d)].
\end{align*}
Therefore, $\widetilde{Q}\in \widetilde{\mathscr{P}}_{\max}$.

On the other hand, for any bounded Borel-measurable function $f(x_1,\ldots,x_{d-1})$, we have
\begin{align*}
& E_{\widetilde{Q}}[(x_d^2\wedge b_d^2)f(x_1,\ldots,x_{d-1})]=E_{\widetilde{Q}_d}[[(x_d^2\wedge b_d^2)f(x_1,\ldots,x_{d-1})] \\
=& \idotsint f(x_1,\ldots,x_{d-1}) (x_d^2\wedge b_d^2) k_{d-1}(x_1,\ldots,x_{d-1};dx_d)\widetilde{Q}_{d-1}(dx_1,\ldots,dx_{d-1})\\
=& \idotsint f(x_1,\ldots,x_{d-1})\big(\alpha_{d-1}\overline{\mu}_d+(1-\alpha_{d-1})\underline{\mu}_d\big)\widetilde{Q}_{d-1}(dx_1,\ldots,dx_{d-1})
 \;\; (\text{by (\ref{eqproofth2.12}})) \\
=& \idotsint f(x_1,\ldots,x_{d-1})\psi_{d-1}(x_1,\ldots,x_{d-1})\widetilde{Q}_{d-1}(dx_1,\ldots,dx_{d-1})\\
=&E_{\widetilde{Q}}\big[f(x_1,\ldots,x_{d-1})\psi_{d-1}(x_1,\ldots,x_{d-1})\big].
\end{align*}
Thus, (\ref{eqconditionE2}) holds.

Finally, by (\ref{eq:identity}), there exists $Q\in \text{core}(\Sbep)$ such that $\widetilde{Q}=Q\bm Y^{-1}$. Hence, from (\ref{eqconditionE2}), (\ref{eqconditionE}) holds.

{\em Step 4}. We prove that for any sequence of continuous functions $\{\varphi_j(x_1,\ldots,x_j), j\ge 1\}$, there exists $P\in \mathscr{P}$ such that
\begin{align}\label{eqconvergence2}
\sum_{k=1}^{\infty} P\bigg(\max_{n\le n_{k+1}}\bigg|\sum_{j=1}^n(Z_j^2-\cSbep[Z_j^2]\vee \varphi_{j-1}(Y_1,\ldots,Y_{j-1})\wedge \Sbep[Z_j^2]\bigg|\ge 2\epsilon_k n_k\bigg)\le M<\infty.
\end{align}

Let $\psi_{j-1}(x_1,\ldots,x_{j-1})=
\cSbep[Z_j^2]\vee \varphi_{j-1}(x_1,\ldots,x_{j-1})\wedge \Sbep[Z_j^2]$.
By Step 2, there exists $Q\in \text{core}(\Sbep)$ such that (\ref{eqconditionE}) holds. Combining with (\ref{eqconvergence}), we have
\begin{align} \label{eq:eqconvergence3}
\sum_{k=1}^{\infty} Q\bigg(\max_{n\le n_{k+1}}\bigg|\sum_{j=1}^n(Z_j^2-\psi_{j-1}(Y_1,\ldots,Y_{j-1}))\bigg|\ge \epsilon_k n_k\bigg)\le M<\infty.
\end{align}
Let $f\in C_{b,Lip}(\mathbb R)$ be such that $I\{|x|\ge 2\}\le f(x)\le I\{|x|\ge 1\}$. Denote
 $$f_k(\bm x)=f\bigg(\max_{ n\le n_{k+1}}\bigg|\sum_{j=1}^n(x_j^{(b_j)})^2-\varphi_{j-1}(x_1,\ldots,x_{j-1}))\bigg|/\epsilon_k n_k\bigg). $$
Then
\begin{align*}
 & E_{Q}\bigg[\sum_{k=1}^{\infty}f_k(\bm Y)\bigg]
 \le \sum_{k=1}^{\infty} Q\bigg(\max_{n\le n_{k+1}}\bigg|\sum_{j=1}^n(Z_j^2-\psi_{j-1}(Y_1,\ldots,Y_{j-1}))\bigg|\ge \epsilon_k n_k\bigg)\le M<\infty.
 \end{align*}
Note that by (\ref{eq:core-real:3ad}) we have
$$ \sup_{P\in \text{core}(\Sbep)}E_P[\varphi(Y_1,\ldots,Y_d)]=\sup_{P\in \mathscr{P}}E_P[\varphi(Y_1,\ldots,Y_d)], \;\;\varphi\in C_b(\mathbb R^d). $$
Now, for each $l$, $\sum_{k=1}^l f_k(\bm x)$ is a bounded continuous function on $\mathbb R^{n_{l+1}}$. Thus
$$ \inf_{P\in \mathscr{P}}E_P[\sum_{k=1}^lf_k(\bm Y)]=\inf_{P\in \text{core}(\Sbep)}E_P[\sum_{k=1}^lf_k(\bm Y)]\le E_Q[\sum_{k=1}^lf_k(\bm Y)]\le M. $$
Therefore, there exists $P_l\in \mathscr{P}$ such that
$$ E_{P_l}[\sum_{k=1}^lf_k(\bm Y)]\le M+1/l. $$
By the weak compactness of $\mathscr{P}$, there exists a subsequence $l^{\prime}\nearrow \infty$ and $P\in \mathscr{P}$ such that $P_{l^{\prime}}$ converges weakly to $P$. From this we obtain
$$ E_{P}[\sum_{k=1}^lf_k(\bm Y)]=\lim_{l^{\prime}\to \infty} E_{P_{l^{\prime}}}[\sum_{k=1}^lf_k(\bm Y)]
\le \liminf_{l^{\prime}\to \infty}E_{P_{l^{\prime}}}[\sum_{k=1}^{l^{\prime}}f_k(\bm Y)]\le M,\;\; \forall l\ge 1. $$
Therefore
 \begin{align*}
 &\sum_{k=1}^{\infty} P\bigg(\max_{ n\le n_{k+1}}\bigg|\sum_{j=1}^n(Z_j^2-\psi_{j-1}(Y_1,\ldots,Y_{j-1}))\bigg|\ge 2\epsilon_k n_k\bigg) \\
 & \;\; \le E_{P}[\sum_{k=1}^{\infty}f_k(\bm Y)]\le M<\infty,
\end{align*}
where $M$ and $\{\epsilon_k\}$ do not depend on the choice of $\varphi_i$s and $P$.

{\em Step 5}. We prove (\ref{eq:prop2.2}).

Suppose $\varphi(\bm x)$ is a continuous function on $\mathbb{R}^{\mathbb N}$. Without loss of generality, assume $\varphi(\bm x)\in [\underline{\sigma}^2,\overline{\sigma}^2]$. Let $\bm x_0=(x_1^0,x_2^0,\ldots)$, $\varphi_j(x_1,\ldots,x_j)=
\varphi(x_1,\ldots,x_j, x_{j+1}^0,x_{j+2}^0,\ldots)$. Then by the continuity of $\varphi(\cdot)$, for any $\bm x=(x_1,x_2,\ldots)$ we have
$$ \varphi_j(x_1,\ldots,x_j)\to \varphi(x_1,x_2,\ldots). $$
Since $\Sbep[Z_j^2]\to \overline{\sigma}^2$ and $\cSbep[Z_j^2]\to \underline{\sigma}^2$, we have
$$ \frac{\sum_{j=1}^n \psi_{j-1}(x_1,\ldots,x_{j-1})}{n}\to \varphi(x_1,x_2,\ldots), $$
where $\psi_0=0$, and $\psi_{j-1}(x_1,\ldots, x_{j-1})=\cSbep[Z_j^2]\vee \varphi_{j-1}(x_1,\ldots, x_{j-1})\wedge \Sbep[Z_j^2]$. Since $\psi_{j}(x_1,\ldots, x_{j})$ is continuous, by Step 4, there exists $P\in \mathscr{P}$ such that (\ref{eqconvergence2}) holds. Therefore
$$ P\bigg(\lim_{n\to\infty}\frac{V_n^2}{n}=\varphi(Y_1,Y_2,\ldots)\bigg)
=P\bigg(\lim_{n\to\infty}\frac{\sum_{j=1}^n \big(Z_j^2-\psi_{j-1}(Y_1,\ldots, Y_{j-1})\big)}{n}=0\bigg)=1. $$

Now assume $\varphi(\bm x)=\varphi(x_1,\ldots,x_d)$ is a Borel-measurable function on $\mathbb{R}^{d}$. Without loss of generality, assume $\varphi(\bm x)\in [\underline{\sigma}^2,\overline{\sigma}^2]$. Define
$$ \psi_j(x_1,\ldots, x_j)=\Sbep[Z_{j+1}^2], \; 0\le j\le d-1, $$
$$\psi_{j}(x_1,\ldots, x_{j})=\cSbep[Z_{j+1}^2]\vee \varphi(x_1,\ldots, x_d)\wedge \Sbep[Z_{j+1}^2], \; j\ge d. $$
Then
$$ \frac{\sum_{j=1}^n \psi_{j-1}(x_1,\ldots,x_{j-1})}{n}\to \varphi(x_1,\ldots,x_d). $$
By Steps 1 and 3, there exists $Q\in \text{core}(\Sbep)$ such that (\ref{eq:eqconvergence3}) holds. Thus
$$ Q\bigg(\lim_{n\to\infty}\frac{V_n^2}{n}=\varphi(Y_1,\ldots, Y_d)\bigg)
=Q\bigg(\lim_{n\to\infty}\frac{\sum_{j=1}^n \big(Z_j^2-\psi_{j-1}(Y_1,\ldots, Y_{j-1})\big)}{n}=0\bigg)=1. $$
The conclusion is proved.
\end{proof}

At last, we remove the condition (\ref{eqLILmomentcondition3}) in Theorem \ref{thLIL2}. It is sufficient to replace Proposition \ref{prop:1}  by the following proposition.

\begin{proposition} \label{prop:3} Assume that there exists a family of probability measures $\mathscr{P}$ on $(\Omega,\sigma(\mathscr{H}))$ such that the sub-linear expectation $\Sbep$ satisfies (\ref{eqexpressbyP}). Assume conditions (\ref{eqLILmomentcondition1}), (\ref{eqLILmeanzerocondition}) and (\ref{eqLILmomentcondition2}) hold. Then for any $P\in \mathscr{P}$, on the event $\big\{V_n/\sqrt{n}\to \zeta\big\}$ we have
\begin{align}\label{eq:prop3.1} Cl\bigg\{ \frac{S_n}{\sqrt{2n\log\log n}}\bigg\}=[-\zeta,\zeta]\;\; a.s. \text{ under } P,
\end{align}
\begin{align}\label{eq:prop3.2}  \lim_{n\to \infty} \max_{k\le n} \frac{S_k}{\sqrt{2k\log\log n}}=\zeta\text{ and } \lim_{n\to \infty} \min_{k\le n} \frac{S_k}{\sqrt{2k\log\log n}}=-\zeta\;\; a.s. \text{ under } P. 
\end{align}
Here, $B$ a.s.  under $P$ on $A$ means that $P(B^cA)=0$. 
 \end{proposition}
 \begin{proof} Let $\Omega_0=\big\{V_n/\sqrt{n}\to \zeta\big\}$. By (\ref{eq:prop2.1}), on $\Omega_0$, $\underline{\sigma}\le \zeta\le \overline{\sigma}$ a.s. under $P$. It is easily seen that, on the event 
 $\Omega_0\bigcap \{\zeta>0\}$, (\ref{eq:prop:1.1}), (\ref{eq:prop:1.2}) and (\ref{eq:prop:1.8}) remain true.  Thus, on the event $\Omega_0\bigcap \{\zeta>0\}$,
 \begin{align}\label{eq:prop:3.1}
& \lim_{n\to \infty} \max_{k\le n}\frac{M_k}{\sqrt{2k\log\log n}}
=\zeta \;\; a.s. \text{ under } P.
\end{align}
On the other hand, on the event  $\Omega_0\bigcap \{\zeta=0\}$,
\begin{align*} 
 &T_n=\sum_{j=1}^n \tau_j = \sum_{j=1}^n E_P[\tau_j|\mathscr{G}_{j-1}]+o(n)\\
  = & \sum_{j=1}^n E_P[(\Delta M_j)^2|M_1,\ldots, M_{j-1}] +o(n) \\
  =& U_n^2+o(n)=V_n^2+o(n)=o(n)\;\; a.s. \text{ under } P.
\end{align*}
Thus, on the event $\Omega_0\bigcap \{\zeta=0\}$,
\begin{align*} 
& \limsup_{n\to \infty} \max_{k\le n}\frac{|M_k|}{\sqrt{2k\log\log n}}
 = \limsup_{n\to \infty} \max_{k\le n}\frac{|W(T_k)|}{\sqrt{2n\log\log n}}\\
 \le & \limsup_{n\to \infty} \max_{0\le s\le o(n)}\frac{|W(s)|}{\sqrt{2n\log\log n}}=0 \;\; a.s. \text{ under } P.
\end{align*}
We conclude that on the event $\Omega_0$, (\ref{eq:prop:3.1}) holds.  On the other hand, (\ref{eq:remainder}) remains true since in Zhang \cite{Zhang2022}, the condition (\ref{eqLILmomentcondition3}) is only used for replacing $n$ by $V_n^2$.  Therefore,   On the event $\Omega_0$, 
 $$ \lim_{n\to\infty}\max_{k\le n} \frac{\sum_{j=1}^k Y_j}{\sqrt{2k\log\log n}}=\zeta \; \; a.s. \text{ under } P.
$$
Similarly,
$$ \lim_{n\to\infty}\min_{k\le n} \frac{\sum_{j=1}^k Y_j}{\sqrt{2k\log\log n}}=-\lim_{n\to\infty}\max_{k\le n} \frac{\sum_{j=1}^k (-Y_j)}{\sqrt{2k\log\log n}}=-\zeta \; \; a.s. \text{ under } P \text{ on } \Omega_0.
$$
Thus (\ref{eq:prop3.2}) is proved. (\ref{eq:prop3.1}) is implied by (\ref{eq:prop3.2}). 
 \end{proof}

\begin{remark} Recently, Gu and Zhang \cite{GuZhang2026a,GuZhang2026b} proved the strong law of large numbers and the law of the iterated logarithm for $m$-dependent random variables under sub-linear expectations. A natural question is whether similar results to Theorem \ref{thLIL2} or (\ref{eq:chenLIL}) hold for $m$-dependent random variables.
\end{remark}

\bigskip
\noindent {\bf Acknowledgements}: { Li-Xin Zhang was supported by the National Key Research and Development Program of China (Grant No. 2024YFA1013502), the National Natural Science Foundation of China (Grant No. U23A2064), and the Zhejiang Provincial Peak Discipline Program (Zhejiang Gongshang University -- Statistics). Yongsheng Song was supported by the National Key Research and Development Program of China (Grant No. 2024YFA1013503) and the National Natural Science Foundation of China (Grant No. 12431017).}



\bigskip

  \clearpage
\end{document}